\documentclass{amsart}
\usepackage{amsmath,amsthm,amsfonts,amssymb,latexsym,mathrsfs,graphicx}

\usepackage{hyperref}

\usepackage{enumerate}
\usepackage[shortlabels]{enumitem}
\usepackage{color}

\newtheorem{theorem}{Theorem}[section]

\newtheorem{lemma}[theorem]{Lemma}

\theoremstyle{definition}

\newtheorem*{defn}{Definition}

\newcommand{\Irr}{{\mathrm {Irr}}}
\newcommand{\Nd}{{\mathfrak{N}_d}}
\newcommand{\N}{{\mathfrak{N}}}
\newcommand{\m}{{\mathfrak{M}}}

\newcommand{\acd}{{\mathrm {acd}}}
\newcommand{\Aut}{{\mathrm {Aut}}}

\newcommand{\PSL}{{\mathrm {PSL}}}

\newcommand{\SL}{{\mathrm {SL}}}

\begin{document}
	
	\title{Erratum to:A generalization of Taketa's theorem on $\rm M$-groups}
	
	\author[Z. Akhlaghi]{Zeinab Akhlaghi}
	\address{Zeinab Akhlaghi, Faculty of Math. and Computer Sci., \newline Amirkabir University of Technology (Tehran Polytechnic), 15914 Tehran, Iran.\newline
		and \newline
		School of Mathematics,
		Institute for Research in Fundamental Science(IPM)
		P.O. Box:19395-5746, Tehran, Iran.}
	\email{z\_akhlaghi@aut.ac.ir}

	\thanks{
		The  author  is supported by a grant from IPM (No. 1401200113).}
	
	\begin{abstract}  In the recent paper  [A generalization of Taketa's theorem on $M$-groups, Quaestiones Mathematicae, (2022), https://doi.org/10.2989/16073606.2022.2081632], we give an upper bound $5/2$ for the average of non-monomial character degrees of a  finite group  $G$, denoted by $\acd_{nm}(G)$, which guarantees the solvability of $G$. Although the result is true,  the   example we gave  to show that the bound is sharp turns out to be  incorrect. In this paper we find  a new bound and we give an example to show that this new bound is sharp.  Indeed,  we prove the solvability of  $G$,  by assuming   $\acd_{nm}(G)< \acd_{nm}(\SL_2(5))=19/7$. 
	\end{abstract}


	\maketitle
	
	\section{Introduction}
	
	This paper is a corrigendum to our paper [A generalization of Taketa's Theorem on $M$-groups, Quaestiones Mathematicae, (2022), https://doi.org/10.2989/16073606.2022.2081632], which will  refer to as \cite{khodom}, hereafter. 
	In \cite{khodom}, denoting by $\acd_{nm}(G)$ the average of non-monomial irreducible characters of $G$, we proved that if $\acd_{nm}(G)< 5/2$, then $G$ is solvable. This result is true, however the bound is not sharp. In fact, after the paper was published, we found  a mistake in  calculating $\acd_{nm}(\rm{S}_5)$  which we  have  miscalculated to be  $5/2$, while it is equal to $10/3>5/2$. Thus, however the proof of the main result in \cite{khodom} is correct, the bound is not sharp. In this paper we aim to to find  the best possible bound and prove the same result for this new bound. Indeed we  prove the solvability of $G$, by assuming  $\acd_{nm}(G)< \acd_{mn}(\SL_2(5))=19/7$.

	Let $N$ be a normal subgroup of $G$ and $\lambda \in \Irr(N)$. Then  $\Irr(G|N)$, $\Irr_{nm}(G|N)$ and $\Irr(G|\lambda)$ denote the set of  irreducible characters of $G$ whose kernels do not contain $N$, the set of the non-monomial  irreducible characters of $G$ whose kernels do not contain $N$ and the set of the irreducible characters of $G$ above $\lambda$, respectively.   By  $\acd_{nm}(G|N)$  we mean   the average  degree of irreducible characters in $\Irr_{nm}(G|N)$. We use the notation $n_d(G)$,  $\Nd(G)$, $n_d(G|N)$ and $\Nd(G|N)$ for the number of irreducible characters of $G$ of degree $d$,  the number of non-monomial irreducible characters of $G$ of degree $d$, the number of irreducible characters of $G$ of degree $d$ whose kernels do not contain $N$ and the number of non-monomial irreducible characters of $G$ of degree $d$ whose kernels do not contain $N$,  respectively.    For the rest of notation, we follow \cite{Isaacs}.

	\section{Main Results}
	In the following we bring up some results and  a generalized definition of a monomial group which is taken from \cite{lmt} and \cite{qianpri}.  
	
	\begin{defn}
		A non-linear character $\chi \in \Irr(G)$ is called a multiply
		imprimitive character (or $m.i$ character for short) induced from the pair $(U, \lambda)$ if there exist a proper subgroup $U$ of $G$ and an irreducible character $\lambda \in \Irr(U)$ such
		that $\lambda ^G=m\chi$
		for some nonnegative integer $m$. 
	\end{defn}
	
	Clearly, if an irreducible character  is not a $m.i$ character, then it is primitive (in \cite[Definition 3.2]{qianpri} they are  called  super-primitive character).  The proof of the following Lemma is  essentially taken from \cite{lmt}. 
	
	\begin{lemma}\label{lmt}
		If $S$ is a non-abelian simple group, then $S$ has a non-linear   irreducible
		character $\chi$ of degree at least $4$ which is extendible to $\Aut(S)$ and $\chi$ is not an $m.i$ character ($ \chi$ is a  super-primitive character). In particular, if $S\not \cong {\rm A}_5$, then $\chi(1)\geq 5$. 
	\end{lemma}
	\begin{proof}
		First, let $S$ be a simple group of Lie type, then using \cite[Lemma 3.8]{lmt} the Steinberg character of $S$  is  not a $m.i$ character and it is  extendable to $\Aut(S)$ (see for example \cite{f}), except for  the cases $S\cong A_5$, $\PSL_2(7)$  or $S_4(3)$. One can check that in those remaining cases  $S$ has an  irreducible character  of degree $4$, $6$ and $2^6$, respectively, which is not a $m.i $ character and it is  extendable to $\Aut(S)$ (see \cite{atlas}), as wanted.
		
		Next, let $S$ be  an alternating group of degree $n\geq 7$. By the proof of \cite[Proposition 4.1]{lmt}, we see that $S$ has an   irreducible character $\chi$ of degree $n-1\geq 6$ which is extendable to $\Aut(S)$ and it is not a $m.i$ character.  
		
		At last, let $S$ be a sporadic simple group or the Tits group. Then according to the the proof of \cite[Proposition 5.1]{lmt} if  for each $S$ listed in the first column of Table 2 of  \cite{lmt} we take  $\chi$ to be the character in the  third column of Table 2 of  \cite{lmt}, then $\chi$ satisfies  the  hypothesis of the  Lemma.    
	\end{proof}
	
	\begin{lemma}(See \cite[Lemma 5]{ema}) \label{ema}
		Let $ N $ be a minimal normal subgroup of $ G $ such that $ N =
		S_{1} \times \dots \times S_{t} $, where $ S_{i} \cong S$, a non-abelian simple group. If $ \sigma \in \Irr(S)$ extends to $ \Aut(S) $, then	$ \sigma \times  \dots \times \sigma \in \Irr(N) $ extends to $ G $.
	\end{lemma}
	
\begin{lemma}\label{qpp}
	Assume  $N\cong S^k$  is a non-abelian  minimal normal subgroup of a finite group $G$, for some simple group $S$ and integer $k$. Let either  $\phi\in \Irr(G/N)$ be primitive and  $\sigma \in \Irr(S)$  be a  super-primitive character; or $\phi\in \Irr(G/N)$ be linear  and  $\sigma \in \Irr(S)$  be primitive character. If $\chi$ is an extension of the product of  all different conjugate of $\sigma$ to $G$, then $\chi\phi$ is primitive.
\end{lemma}

\begin{proof}
	Note that $N$ and $G$ satisfies \cite[Hypothesis 1.1]{qianpri} and $\chi$ is an extension of $\theta$ the product of all distinct $G$-conjugate of $\sigma$ (see the explanation before \cite[Lemma 3.1]{qianpri}). By  applying    \cite[Corollary 3.4]{qianpri},  $\chi\phi$ is  primitive. 
\end{proof}

\begin{lemma}\label{qp}(See \cite[Lemma 2.8]{qp}.)
	Let $1<N$  be a normal subgroup of $G$ with $N = N'$, and suppose that a non-trivial $ \theta \in \Irr(N)$
	extends to $\chi \in \Irr(G)$. If $\chi=\lambda^G$, where $\lambda \in  \Irr(H)$ for some subgroup $H$ of $G$, then the following results are
	true:\\

	(1) $HN = G$, and $\theta=(\lambda_{H\cap N})^N$,  where $\lambda_{H\cap N}$ is irreducible. 
	
	(2) If in addition $\theta(1)$ is minimal among all non-trivial irreducible character degrees of $N$, then $\chi$ is primitive.
\end{lemma}


\begin{theorem}
	Let $G$ be a finite group.  If $\acd_{nm}(G)< 19/7$, then $G$ is solvable.
\end{theorem} 
\begin{proof}
	Assume, on the contrary,   $G$ is an example  with minimal order, such that $G$ is non-solvable and  $\acd_{nm}(G)< 19/7$.  Thus,
	$$\acd_{nm}(G)=\frac{\sum \limits_{\chi\in \Irr_{nm}(G)}\chi(1)}{|\Irr_{nm}(G)|}<  19/7.$$ 
	
	Then, $\sum \limits_{\chi\in \Irr_{mn}(G)}\chi(1)=\sum\limits_{d\geq 1}d\Nd(G)$ and $|\Irr_{mn}(G)|=\sum\limits_{d\geq 1}\Nd(G)$.  So by the above inequality we have
	
	$$\sum\limits_{d\geq 3}(7d-19)\Nd(G)< 12 \N_1(G)+ 5\N_2(G). \   \   \   \  \  \  \  \  (*)$$
	
	First, we claim that there is no non-solvable minimal normal subgroup of $G$ contained in $G'$. On the contrary, let $M\leq G'$ be a  non-solvable minimal normal subgroup of $G$. Then $M$ is a direct product of $k$ copies of a  non-abelian finite simple group $S$, for some integer $k$. By the hypothesis $M$ is contained in the kernel of every linear character of $G$. We show that $M$ is contained in the kernel of  every irreducible  character of $G$  of degree $2$. Let $\chi\in \Irr(G)$ such that $\chi(1)=2$. Since, non-abelian finite simple groups do not have any irreducible character of degree $2$ and the only linear character of a simple group is the principle character, then  $\chi_M=2.1_M$. Therefore $M$ lies in the kernel of $\chi$, as wanted. Hence $n_d(G)=n_d(G/M)$, for  $d=1,2$ and so $\N_d(G)=\N_d(G/M)$
	for $d=1,2$.

	By Lemmas \ref{lmt},  \ref{ema} and \ref{qpp}, $M$ has a primitive irreducible character $\theta$ with degree  $d_0\geq 4$ which is extendable to  a primitive irreducible character of  $G$.  
	Note that if $\phi\in \Irr_{nm}(G/M)$ and $\phi(1)=p$ for some prime $p$ then $\phi$ is primitive. Hence the number of primitive characters of degree $p$ of $G/M$, for some prime $p$, is $\N_p(G/M)$.  
	Then, by Gallagher's theorem (see \cite[Corollary 6.17]{Isaacs}) and Lemma \ref{qpp}, we have $\N_1(G)+\N_2(G)=\N_1(G/M)+ \N_2(G/M)\leq \N_{d_0}(G|M)+ \N_{2d_0}(G|M)$ and so   $\N_1(G)+\N_2(G) \leq \sum\limits_{d_0\mid d}\N_{d}(G|M)$.
	
	
	Therefore,  $$\N_1(G)+ \N_2(G)\leq \sum\limits_{d_0\mid d}\N_d(G).$$ 
	
	Hence,  
	$$\sum\limits_{d_0\mid d}(7d-19)\N_{d}(G) \geq (7d_0-19) (\N_2(G) +\N_1(G)). \ \ \ \ \ (**)$$
	If $M\not\cong \rm{A}_5$, then $d_0\geq 5$ by Lemma \ref{lmt}, and the non-equality  $(**)$ contradicts $(*)$.		 
	Now, let $M\cong \rm{A}_5$ and set $C=C_G(M)$. We know that $MC/C\leq G/C \leq \Aut(MC/C)\cong \rm{S}_5$. First, let $G=MC\cong C\times M$. Then, $M$ contains three non-linear  primitive characters, two of degree $3$ and one of degree $4$. Note that in this case $(**)$ holds with $d_0=4$, which means that 
		$$\sum\limits_{4\mid d}(7d-19)\N_{d}(G) \geq 9 (\N_2(G) +\N_1(G)).$$
	On the other hand, by Lemma \ref{qpp}, the extensions of irreducible  characters of degree $3$ of $M$ to $G$ are primitive, yielding that  $\N_3(G)\geq \N_3(G|M)=2\N_1(G/M)=2\N_1(G)$. 	Hence,    
		$$\sum\limits_{d\geq 3}(7d-19)\N_{d}(G)\geq \sum\limits_{d\geq 4}(7d-19)\N_{d}(G) + 2\N_3(G) \geq 9 (\N_2(G) +\N_1(G))+4\N_1(G)\geq 5\N_2(G)+12\N_1(G),$$
	which  is contradicting  $(*)$. So, we may assume $G\not=CM$,  implying that $G/C\cong \rm{S}_5$. Recall that $M$ contains a character $\chi\in \Irr(M)$ of degree $5$ which is extendable to $G$. We show that all extensions of $\chi$ to $G$ are primitive. On the contrary, assume $\chi_0\in \Irr(G)$ is an extension of $\chi$ that is not primitive, which means that there exists a subgroup $H< G$ and a linear character $\eta\in \Irr(H)$ such that $\chi_0=\eta^G$. Remark that $G/C\cong {\rm S}_5$ has a primitive extension of $\chi$, say $\psi$. By Gallagher's theorem \cite[Theorem 6.17]{Isaacs}, $\psi=\chi_0\lambda$ for some linear character $\lambda\in \Irr(G/M)$. As $\psi=\chi_0\lambda=\eta^G\lambda=(\eta\lambda_H)^G$, (see \cite[Problem 5.3]{Isaacs}), we get that $\psi$ is not primitive, a contradiction. Hence,  all extensions of $\chi$ to $G$ are primitive. 	
	 Thus, $\N_5(G)\geq \N_5(G|M)=\N_1(G/M)=\N_1(G).$ Again using, $(**)$, we have  $$\sum\limits_{ d\geq 4}(7d-19)\N_{d}(G)\geq \sum\limits_{4\mid  d}(7d-19)\N_{d}(G) + 16\N_5(G) \geq 9 (\N_2(G) +\N_1(G))+16\N_1(G)\geq 5\N_2(G)+12\N_1(G),$$       
	which is a contradiction. 

	Therefore, our claim is proved. Hence, we  may assume  that every  minimal normal subgroup of $G$ contained in $G'$ is solvable.  Let $M\unlhd G$   be minimal such that $M$ is non-solvable. Notice that  $M$ is a perfect group contained in the last term of derived series of $G$. Let $T\leq M$, such that $T$ is a minimal normal subgroup of $G$. In addition, if $[M, R]\not =1$, we assume $T\leq [M,R]$, where $R$ is the radical solvable of $M$.
	Therefore $T\leq M'\leq G'$, so $T$ is solvable  and then  $G/T$ is non-solvable. As $G$ is a counterexample of minimal order, we  have $\acd_{nm}(G/T)
	\geq  19/7$  and so it  follows by arguing
	exactly as in the first paragraph of the proof that
	
	$$\sum\limits_{
		d\geq 3}(7d-19)\N_d(G/T) 
	\geq  \N_2(G/T) + 12\N_1(G/T).\  \  \  \  (**)$$  
	Noting $\N_1(G/T)=\N_1(G)$, we get that 
	$\N_2(G/T) <\N_2(G)$ from $(*)$ and $(**)$.
	Hence $\Irr_{nm}(G|T)$ contains a character of degree $2$, say $\chi$.  If $K=\ker(\chi)$, then $G/K$ is a primitive linear
	group of degree 2 (see \cite[Chapter 14]{Isaacs}). By the classification of the non-solvable primitive
	linear groups of degree 2 (see \cite[Theorem 14.23]{Isaacs})  	
	we have $G/C\cong A_5$, where $C/K=Z(G/K)$. This implies that  $G=MC$. 				  
	Recall that
	$M/(M\cap C)\cong MC/C= G/C$ and
	$M\cap C\unlhd M$ is a proper subgroup of $M$. Therefore,
	$ M\cap C $  is a subgroup of radical solvable subgroup of $M$ by the
	minimality of $M$. Since $M/(M\cap C)$ is simple, we obtain that $M\cap C=R$, where $R$ is the radical solvable subgroup of $M$,
	and hence $R\leq C$. Thus $[M, R] \leq K$. But $ T \not\leq K$, so we have $ T \not\leq
	[M, R]$. By the choice of $T$, we have $R = {\bf Z}(M)$. Therefore,
	$M$ is a perfect central cover of the simple group $M/{\bf Z}(M) = M/(M \cap C) \cong G/C$.
	Since $C$ and $M$ are both normal in $G$, we have $[M, C]  \leq C \cap M =
	{\bf Z}(M)$, and so $[C, M, M] = [M, C, M] = 1$. By the three subgroups lemma, we deduce that
	$[M, M, C] = 1$ and hence $[M, C] = 1$ as $M$ is perfect. We conclude that $G = MC$ is
	a central product with a central subgroup $ M \cap  C =  {\bf Z}(M) \not = 1$. 				  
	Thus, by the choice of $T$, we get that $T= M\cap C={\bf Z}(M)$. As, ${\bf Z}(M)$ lies in the Schur multiplier of $M/{\bf Z}(M)\cong \rm A_5$, we deduce that ${\bf Z}(M)=T\cong C_2$.	
	So  $M\cong \SL_2(5)$ and  $\Irr(G|T)= \Irr(G|\lambda)$, where $\lambda$ is the only  non-trivial character of $T$. 
	Recall that  $ \Irr(M|\lambda)$ contains two  primitive characters of degree $2$, one primitive character  of degree $4$, and one character of degree $6$  and   $\Irr_{nm}(G|M)$ contains a primitive character $\chi$ of degree $2$, which is an  extension of one of the irreducible characters of degree $2$ of $M$.  By \cite[Lemma 2]{moreto}, $\chi(1)=\beta(1)\alpha(1)$ where $\beta\in \Irr(C|T)$ and $\alpha\in \Irr(M|T)$. As $\Irr(M|T)$ does not contain any linear character, $\beta$ is a linear character which means $\lambda$ extends to $C$. Applying \cite[Lemma 2]{moreto}, for every $\psi\in \Irr(M|T)$, we have  $\Irr(G|T)$ contains a character  of degree $\psi(1)\beta(1)$, which clearly is the extension of   $\psi$. By Lemma \ref{qp}(1), we deduce that every extension of $\psi$ is primitive, if $\psi(1)\in \{2,4\}$. Therefore, $ \N_4(G|T)=\N_1(G/M)=n_1(G/M)$ and $ \N_2(G|T)=2\N_1(G/M)=2n_1(G/M)$. Then  
	
	$$\acd_{nm}(G|T)=\frac{4n_1(G/M)+ 4n_1(G/M) +\sum\limits_{d\geq 6}d\N_d(G|T)}{2n_1(G/M)+ n_1(G/M) +\sum\limits_{d\geq 6}\N_d(G|T)}=\frac{8n_1(G/M) +\sum\limits_{d\geq 6}d\N_d(G|T)}{3n_1(G/M) +\sum\limits_{d\geq 6}\N_d(G|T)}.$$
	
	On the other hand,  $G/T\cong C/T\times M/T\cong C/T\times \PSL_2(5)$. Then using Lemma \ref{qp}(1) all irreducible characters $\lambda\times \mu\in \Irr(C/T\times \PSL_2(5))$ are non-monomial, provided that  $\mu \in \Irr(\PSL_2(5))$  is a  non-monomial  character and $\lambda$ is a linear character of $C/T$, which means $\mu$  is one of those irreducible characters of  $\PSL_2(5)$ of  degree $3$ or $4$. Also, by \ref{qpp}, if $\mu$ is the only super-primitive character of $M/T$ of degree $4$ and $\lambda$ is a primitive character of $C/T$, then $\lambda\times \mu$ is primitive.  We denote by $\m_d(G/T)$  the number of non-monomial   irreducible characters of  $G/T$ of degree $d$ in form of $\lambda\times \mu$, where either $\mu\in \Irr(\PSL_2(5))$ has degree  $1$ or $5$;  $\mu$ has degree $3$ and $\lambda$ is not linear; or $\mu$ has order $4$ and $\lambda(1)>2$. Clearly $\m_d(G/T)=\N_d(G/T)=\N_d(G/M)$ for $d=1,2$ and $\N_1(G/M)=n_1(G/M)$. Therefore, by the above argument,   
			$$\acd_{nm}(G)=\frac{\sum\limits_{d\geq 1}d\N_d(G|T)+\sum\limits_{d\geq 1}d\N_d(G/T)}{|\Irr_{nm}(G|T)|+|\Irr_{nm}(G/T)|}=$$$$
		\frac{10n_1(G/M)+ 8\N_2(G/M)+ \sum\limits_{1\leq d\leq 2}d\N_d(G/M)+ \sum\limits_{d\geq 3}d\m_d(G/T) +8n_1(G/M)+\sum\limits_{d\geq 6}d\N_d(G|T)}{3n_1(G/M)+ \N_2(G/M)+\sum\limits_{1\leq d\leq 2}\N_d(G/M)+  \sum\limits_{ d\geq 3}\m_d(G/T)+3n_1(G/M)+ \sum\limits_{d\geq 6}\N_d(G|T)}=$$$$
		\frac{19n_1(G/M)+ 10\N_2(G/M)+  \sum\limits_{d\geq 3}d\m_d(G/T) +\sum\limits_{d\geq 6}d\N_d(G|T)}{7n_1(G/M)+ 2\N_2(G/M)+  \sum\limits_{ d\geq 3}\m_d(G/T)+ \sum\limits_{d\geq 6}\N_d(G|T)}.$$
		 Therefore, 
	$$7(19n_1(G/M)+ 10\N_2(G/M)+  \sum\limits_{d\geq 3}d\m_d(G/T) +\sum\limits_{d\geq 6}d\N_d(G|T))$$
	$$\geq  7(19n_1(G/M)+ 10\N_2(G/M)+ \sum\limits_{d\geq 3}3\m_d(G/T) +\sum\limits_{d\geq 6}6\N_d(G|T))$$
		$$\geq 19(7n_1(G/M)+ 2\N_2(G/M)+  \sum\limits_{ d\geq 3}\m_d(G/T)+ \sum\limits_{d\geq 6}\N_d(G|T)),$$
	which means    $\acd_{nm}(G)\geq  19/7$. This contradiction proves the theorem. 
\end{proof}


{\bf Acknowledgment. } The author would like to thank  Neda Ahanjide, for bringing the error in \cite{khodom} to her attention.


\begin{thebibliography}{1}
	
	\bibitem{khodom} Z. Akhlaghi, A generalization of Taketa's theorem on $M$-groups, Quaestiones Mathematicae, https://doi.org/10.2989/16073606.2022.2081632, (2022). 
	
	
	
	
	
	\bibitem{ema} 
	M. Bianchi, D. Chillag, M. L. Lewis, E. Pacifici,
	\newblock 'Character degree graphs that are complete graphs',
	\newblock { Proc. Amer. Math. Soc.}
	\textbf{135}(3) (2007), 
	671–-676.
	
	\bibitem{atlas} J. H. Conway, R. T. Curtis, S. P. Norton,  R.A. Parker and
	R.A. Wilson, {\it Atlas of finite groups}, Oxford University Press,
	London, (1984).
	
	
	\bibitem{f} W. Feit, Extending Steinberg characters, Linear algebraic groups and their representations, Contemp. Math. 153 (1993), 1–9.
	
	\bibitem{Isaacs} I. M. Isaacs, {Character Theory of Finite Groups}, New York NY: Academic Press 1976.
	
	Proc. Amer. Math. Soc. 91 (2) (1984), 192–194.
	
	finite group, Israel J. Math. 197 (2013), 55–67.
	
	\bibitem{moreto} A. Moret\' o, H. N. Neguyen, On the average character degree of finite groups, Bull. Lond. Math. Soc., 46 (2014), 454-462. 
	
	
	\bibitem{lmt} T. Le, J. Moori and H. P. Tong-Viet, On a  generalization of M-group, J. Algebra, 374
	(2013), 27–41.
	
	
	
	
	
	
	
	
	
	\bibitem{qp} G. Qian, Two results related to the solvability of M-groups, J. Algebra 323 (2010), 3134-3141. 
	
	\bibitem{qian} G. Qian, On the average character degree and the average class size in finite groups, J. Algebra 423 (2015), 1191-1212. 
	
	\bibitem{qianpri} G. Qian, Nonsolvable groups with few primitive character degrees, J.  Group Theory  21 (2) (2018), 295-318.
	
	
	
	
\end{thebibliography}
\end{document}